\documentclass[11pt,oneside,reqno]{amsart}
\usepackage{amssymb,amsmath,amsthm}
\usepackage{a4wide,comment,enumitem,url}
\usepackage{xcolor}
\usepackage{enumitem}
\usepackage{bm}
\usepackage{float}

\usepackage{graphicx}
\usepackage[draft]{hyperref}
\usepackage{amsmath,amsopn,amssymb,amsfonts,stmaryrd}
\usepackage{verbatim}
\usepackage{amsthm}
\usepackage{mathtools}
\usepackage{color}
\usepackage{enumitem}
\usepackage[framemethod=TikZ]{mdframed}
\usepackage{bbm}
\usepackage{mathrsfs}
\usepackage{booktabs}
\usepackage{caption}
\usepackage{cancel}
\usepackage{tensor}
\usepackage{cleveref}

\renewcommand{\H}{\mathbb{H}}

\newcommand{\CA}{\mathcal{A}}

\newcommand{\sgn}{\mbox{sgn}}

\newcommand{\SL}{\mathrm{SL}}

\newcommand{\N}{\mathbb N}
\newcommand{\C}{\mathbb C}
\newcommand{\Q}{\mathbb Q}

\theoremstyle{plain}
\newtheorem{thm}{Theorem}[section]

\newtheorem{lem}[thm]{Lemma}
\newtheorem{prop}[thm]{Proposition}
\newtheorem{conj}[thm]{Conjecture}

\theoremstyle{definition}

\newtheorem*{rems}{Remarks}

\numberwithin{equation}{section}
\numberwithin{thm}{section}

\setlist[enumerate]{leftmargin=*,label=\rm{(\arabic*)}}

\renewcommand{\sgn}{\textnormal{sgn}}

\def\d{\delta}
\def\h{\eta}

\def\l{\lambda}

\def\z{\zeta}

\def\t{\tau}

\def\d{\delta}
\def\h{\eta}

\def\l{\lambda}

\def\z{\zeta}

\def\t{\tau}

\newcommand{\re}{{\rm Re}}

\renewcommand{\sgn}{{\rm sgn}}
\newcommand{\R}{\mathbb R}
\newcommand{\Z}{\mathbb Z}

\setlist[itemize]{noitemsep, topsep=0pt}

\allowdisplaybreaks

\makeatletter
\newcommand{\vast}{\bBigg@{2}}
\newcommand{\Vast}{\bBigg@{5}}
\makeatother

\renewcommand{\pmod}[1]{\ \left( \mathrm{mod} \, #1 \right)}
\newcommand{\Pmod}[1]{\ ( \mathrm{mod} \, #1 )}

\title{On a sign-change conjecture of Schlosser and Zhou}
\author{Kathrin Bringmann}
\address{University of Cologne, Department of Mathematics and Computer Science, Weyertal 86-90, 50931 Cologne, Germany}
\email{kbringma@math.uni-koeln.de}
\author{Bernhard Heim}
\address{University of Cologne, Department of Mathematics and Computer Science, Weyertal 86-90, 50931 Cologne, Germany}
\email{bheim@uni-koeln.de}
\author{Ben Kane}
\address{The University of Hong Kong, Department of Mathematics, Pokfulam, Hong Kong}
\email{bkane@hku.hk}

\begin{document}
\date{\today}
\begin{abstract}
	In this paper, we investigate the signs changes of Fourier coefficients of infinite products of $q$-series of Rogers--Ramanujan type. In particular, we prove a conjecture made by Schlosser--Zhou pertaining to such sign changes for products of modulus $10$.
\end{abstract}
\maketitle
\section{Introduction and statement of results}
Motivated by his famous three conjectures for finite products, for a prime $p$, Borwein \cite{Borwein} investigated the infinite products (now called Borwein products)
\[
G_p(q):=\prod_{n\geq 1} \frac{1-q^n}{1-q^{pn}}.
\]
 Andrews \cite[Theorem 2.1]{And} showed that the signs of the coefficients of $G_p$ are periodic with period $p$, which was independently proven in unpublished work of Garvan and Borwein by different methods (see the remark following \cite[Theorem 2.1]{And}). In \cite{Z}, for certain $\delta\in\R$ Schlosser and Zhou investigated the periodicity of the signs of the Fourier coefficients of $G_p(q)^{\delta}$. Schlosser and Zhou primarily focused on the case $p=3$ \cite{Z}, obtaining that the signs change with period $3$ for $n\geq 158$ and $0.227\leq \delta\leq 2.9999$. In \cite[Corollary 5]{Z}, they proved such results by obtaining formulas for the Fourier coefficients of $G_p(q)^{\delta}$ via the Circle Method, which in turn arise from the fact that $G_p$ satisfies modular properties. They further noted that one can obtain exact formulas for these Fourier coefficients by following the work of Rademacher \cite{Rademacher} and Zuckerman \cite{Zuckerman}. Such exact formulas have a long history. Extending work of Hardy and Ramanujan \cite{HardyRamanujan}, Rademacher \cite{Rademacher} obtained an exact formula for the number of partitions of $n$, and subsequent work of Rademacher--Zuckerman \cite{RademacherZuckerman}, Zuckerman \cite{Zuckerman}, and Ono and the first author \cite{BringmannOno} yielded exact formulas for Fourier coefficients of nonpositive weight weakly holomorphic modular forms i.e., meromorphic modular forms whose poles lie at the cusps. Using such exact formulas, the method from \cite{Z} follows a well-known approach for determining the signs of Fourier coefficients. The exact formula naturally splits into a main asymptotic term and an error term. The main asymptotic term overwhelms the error terms for $n$ sufficiently large, and hence the sign necessarily agrees with the sign of the main asymptotic term for sufficiently large $n$.

 Schlosser and Zhou \cite{Z} conjectured a number of further periodic sign patterns for Fourier coefficients of powers of other Borwein products and also powers of some functions built from shifted products of the shape (with $m$, $M\in\N$)
\[
\prod_{\substack{n\geq 1\\ n\equiv m\pmod{M}}} \left(1-q^n\right).
\]
In this paper, we prove one of these conjectures. Define
\begin{equation*}
	Q_{10}(q) := \frac{\left(q,q^9;q^{10}\right)_\infty}{\left(q^3,q^7;q^{10}\right)_\infty},
\end{equation*}
where $(a_1,\dots,a_\ell;q)_{n}:=(a_1;q)_{n}\cdots(a_{\ell};q)_{n}$ with $(a;q)_n:=\prod_{j=0}^{n-1}(1-aq^j)$, $n\in\N_0\cup\{\infty\}$. For $\delta\in\Z$, we write
\[
Q_{10}^{\delta}(q)=:\sum_{n\geq 0} c_{\delta}(n)q^n.
\]
The following states Conjecture 22 of \cite{Z}.
\begin{conj}\label{C:MainConj}
	The coefficients of $Q_{10}^\d(q)$ satisfy for $\d=1$ the sign pattern $+-++--+--+$ and for $\d=-1$ the sign pattern $++++-----+$. That is to say,
	\begin{align*}
		\sgn\left(c_{1}(n)\right)&=
		\begin{cases}
			1&\text{if }n\equiv 0,2,3,6,9\pmod{10},\\
			-1&\text{otherwise},
		\end{cases}\\
		\sgn\left(c_{-1}(n)\right)&=\begin{cases}1&\text{if }n\equiv 0,1,2,3,9\pmod{10},\\
		-1&\text{otherwise}.
	\end{cases}
	\end{align*}
\end{conj}
In this paper, we prove Conjecture \ref{C:MainConj}.
\begin{thm}
	\Cref{C:MainConj} is true, except for
\begin{align*}
	c_1(n)&=0\text{ for }n\in\{2,5,7,9,15,17,22,27,37,47\},\\
	c_{-1}(n)&=0\text{ for }n\in\{3,4,5,6,9,13,19,23,29,39\}.
\end{align*}
\end{thm}
\begin{rems}
\noindent

\noindent
\begin{enumerate}
\item  In particular, $Q_{10}^{\pm 1}(q)$ has only finitely-many vanishing Fourier coefficients.
Several authors assign to a vanishing coefficient both signs, due to the delicate issue of proving that a coefficient vanishes (we refer to the well-known Lehmer conjecture). With this weaker definition of a sign the conjecture of Schlosser and Zhou is true for all $n\in\N_0$.

\item	While $G_p$ is an eta-quotient, up to a power of $q$ in front, the function $Q_{10}$ is not, so the modular properties of $Q_{10}^{\delta}$ differ from those of $G_p^{\delta}$. As such, a key step in proving \Cref{C:MainConj} is the usage of Jacobi theta functions to establish the modular properties, and the exact formulas that one obtains are of a different flavor.
\item Some of the other conjectures in \cite{Z} involve similar quotients which have the same shape as $Q_{10}$. It is hence likely that the first step towards proving Conjectures 18, 21, and 24 of \cite{Z} involve writing these functions as quotients of Jacobi theta functions.
\end{enumerate}
\end{rems}
The paper is organized as follows: In Section \ref{sec:prelim}, we recall basic facts concerning weakly holomorphic modular forms, the Jacobi theta function, the Dedekind $\eta$-function, Kloosterman sums, and Bessel functions. In Section \ref{sec:exact}, we obtain exact formulas for the counting functions of interest. Section \ref{sec:mainterm} is devoted to detecting the main term contributions.
In Section \ref{sec:Kloosterman} we estimate certain Kloosterman-type sums. Section \ref{sec:error} is devoted to estimate the error term.
\section*{Acknowledgements}
The first author has received funding from the European Research Council (ERC) under the European Union's Horizon 2020 research and innovation programme (grant agreement No. 101001179). The third author was supported by grants from the Research Grants Council of the Hong Kong SAR, China (project numbers HKU 17314122 and HKU 17305923). 
The authors thank Guonio Han for helpful conversations.
\section{Preliminaries}\label{sec:prelim}
\subsection{Weakly holomorphic modular forms}
We briefly introduce modular forms, but refer the reader to \cite{O} for more details.  A function on the complex upper half-plane $F:\H\to\C$ satisfies \begin{it}modularity of weight $\kappa\in\Z$ on a congruence subgroup
$\Gamma\subseteq \SL_2(\Z)$\end{it} \textit{ with multiplier $\chi$} if for every $\gamma=\left(\begin{smallmatrix}a&b\\ c&d\end{smallmatrix}\right)\in\Gamma$ we have
\[
F|_{\kappa}\gamma  = \chi(\gamma) F.
\]
Here the weight $\kappa$ \begin{it}slash operator\end{it} is defined for $\tau\in\H$ by 
\[
F\big|_{\kappa}\gamma(\tau):=(c\tau+d)^{-\kappa} F(\gamma\tau).
\]
We call the equivalence classes of $\Gamma\backslash (\Q\cup\{i\infty\})$ the \begin{it}cusps of $\Gamma$\end{it}. For each cusp $\varrho$, we choose a representative\footnote{Throughout, we assume that $0\leq h<k$ and $\gcd(h,k)=1$.}
$\frac{h}{k}\in\Q$ and $M_{\varrho}\in\SL_2(\Z)$ such that $M_{\varrho}(i\infty)=\frac{h}{k}$; we abuse notation and also refer to $\frac{h}{k}$ as a cusp. If a holomorphic function $F$ satisfies weight $\kappa$ modularity on $\Gamma$ with some multiplier $\chi$, then for every cusp $\varrho$ of the function $F_{\varrho}:=F|_{\kappa} M_{\varrho}$ is invariant under the transformation $\tau\mapsto \tau+\sigma_{\varrho}$ for some $\sigma_{\varrho}\in\N$ and hence has a Fourier expansion
\[
F_{\varrho}(\tau)= \sum_{n\in\Z} c_{F,\varrho}(n) q^{\frac{n}{\sigma_{\varrho}}}\qquad(q:=e^{2\pi i \tau}).
\]
We drop $\varrho$ from the notation if $\varrho=i\infty$. We call $F$ a \begin{it}weight $\kappa$ weakly holomorphic modular form on $\Gamma$ with multiplier  $\chi$\end{it} if $F$ satisfies weight $\kappa$ modularity on $\Gamma$ with multiplier $\chi$, is holomorphic on $\H$, and for each cusp $\varrho$ there exists $n_0$ such that $c_{F,\varrho}(n)=0$ for $n<n_0$. We furthermore call the terms in the Fourier expansion with $n<0$ the \begin{it}principal part\end{it} of $F$ at the cusp $\varrho$.
\subsection{Special modular forms}
We define the {\it Jacobi theta function} (throughout $w\in \C$)
\begin{equation*}
	\vartheta(w;\t) := \sum_{n\in\Z+\frac12} q^{\frac{n^2}2} e^{2\pi in\left(w+\frac12\right)}.
\end{equation*}
We have the Jacobi triple product formula ($\z:=e^{2\pi iw}$ throughout)
\begin{equation}\label{E:triple}
	\vartheta(w;\t) = -iq^\frac18 \z^{-\frac12} (q;q)_\infty (\z;q)_\infty \left(\z^{-1}q;q\right)_\infty.
\end{equation}
The well-known transformation laws of $\vartheta$, going back to Jacobi, may be found for example in 
\cite[Chapter I, Section 11]{Mumford} or
\cite[Proposition 1.3]{Zwegers}. 
To state the transformation laws of $\vartheta$ and its growth towards the cusps for $\re(z)>0$, $k\in\N$, $h\in\Z$ with $\gcd(h,k)=1$, and $h'\in\Z$ with $hh'\equiv -1\Pmod{k}$, we let $\omega_{h,k}$ be defined through\footnote{For the exact shape of $\omega_{h,k}$, see \cite[(5.2.4)]{AndrewsPartitions}.} 
\begin{equation*}
	\eta\left(\frac1k(h+iz)\right) = e^{\frac{\pi i}{12k}(h-h')} \omega_{h,k}^{-1} z^{-\frac12} \eta\left(\frac1k\left(h'+\frac iz\right)\right).
\end{equation*}
\begin{lem}\label{lem:thetaprops}
\noindent

\noindent
	\begin{enumerate}
		\item
		Suppose that $w,z\in\C$ with $\re(z)>0$, $k\in\N$, $h\in\Z$ with $\gcd(h,k)=1$, and $h'\in\Z$ with $hh'\equiv -1\Pmod{k}$. Then we have
		\begin{equation*}
			\vartheta\left(w;\frac1k(h+iz)\right) = e^{\frac{\pi i}{4k}(h-h')} e^{\frac{3\pi i}{4}} \omega_{h,k}^{-3} \sqrt{\frac iz} e^{-\frac{\pi  kw^2}z} \vartheta\left(\frac{iw}z;\frac1k\left(h'+\frac iz\right)\right).
		\end{equation*}
		\item
		Suppose that $w\in\C$, $\tau\in\H$, and $\lambda,\mu\in\Z$. Then
		\begin{equation*}
			\vartheta(w+\l\t+\mu;\t) = (-1)^{\l+\mu} q^{-\frac{\l^2}2} \z^{-\l} \vartheta(w;\t).
		\end{equation*}
		\item For $0\le a<1$ and $b\in\R$, we have
		\begin{equation*}
			\vartheta(a\t+b;\t) = -ie^{-\pi ib} q^{\frac18-\frac a2} \left(1+O\left(q^{\min\{a,1-a\}}\right)\right).
		\end{equation*}
	\end{enumerate}
\end{lem}

\subsection{Zuckerman's exact formula}\label{sec:Zuckerman}
To state the exact formula of Zuckerman \cite[Theorem 1]{Zuckerman}, which has been extended to a larger class of functions which include weight zero weakly holomorphic modular forms by Ono and the first author \cite[Theorem 1.1]{BringmannOno}, let $I_\alpha$ denote the usual $I$-Bessel function. 

\begin{thm}\label{Thm:Zuckerman}
Let $\kappa\leq 0$. Suppose that  $F$ is a weakly holomorphic modular form of weight $\kappa$ on a congruence subgroup $\Gamma$ with transformation law
$$
F\left(\frac1k(h+iz)\right) = \chi(\gamma_{h,k})(-iz)^{-\kappa}   F_{\varrho}\left(\frac1k\left(h'+\frac iz\right)\right)
$$
for some multiplier $\chi: \SL_2(\Z)\to\C$ and where $\gamma_{h,k}  := \left( \begin{smallmatrix} h & \beta \\ k & -h' \end{smallmatrix} \right) \in \operatorname{SL}_2(\mathbb{Z})$. Assume that $F$ has the Fourier expansion at $i\infty$, $F(\tau) = \sum_{n\gg-\infty} a(n)q^{n+\alpha}$, and the Fourier expansions at each $0 \leq \frac{h}{k} < 1$, $F|_{\kappa}\gamma_{h,k}(\tau) = \sum_{n \gg -\infty} a_{h,k}(n) q^{\frac{n + \alpha_{h,k}}{c_{k}}}$. Then for $n + \alpha > 0$, we have
	\begin{multline*}
		a(n) =  2\pi (n+\alpha)^{\frac{\kappa-1}{2}} \sum_{k\ge1} \dfrac{1}{k} \sum_{\substack{0 \leq h < k \\ \gcd(h,k) = 1}}\chi(\gamma_{h,k}) e^{- \frac{2\pi i (n+\alpha) h}{k}}
		\\
		\times \sum_{m+\alpha_{h,k} \leq 0} a_{h,k}(m) e^{ \frac{2\pi i}{k c_{k}} (m + \alpha_{h,k}) h' } \left( \dfrac{\lvert m +\alpha_{h,k} \rvert }{c_{k}} \right)^{ \frac{1 - \kappa}{2}} I_{-\kappa+1}\left( \dfrac{4\pi}{k} \sqrt{\dfrac{\lvert m +\alpha_{h,k} \rvert(n + \alpha)}{c_{k}}} \right).
	\end{multline*}
\end{thm}

\subsection{Kloosterman sums}\hspace{0cm}
We require the following {\it Kloosterman sums}
\begin{equation}\label{eqn:Kloosterman}
	K_k(n,m) := \sum_{\substack{h\Pmod k^*\\hh'\equiv-1\Pmod k}} e^{\frac{2\pi i}{k}\left(nh+mh'\right)},
\end{equation}
where $h\pmod{k}^*$ means that $h$ only runs through coprime element (mod $k$).
Letting $d(n)$ denote the number of divisors of $n$, the following bound for the absolute value of these Kloosterman sums is well-known by work of Weil \cite{Weil} (see \cite[Corollary 11.12]{IwaniecKowalski}).
\begin{lem}\label{lem:KloostermanBound}
	Suppose that for $k\in\N$, $n$, $m\in\Z$. Then we have
	\[
		\left|K_k(n,m)\right| \le \sqrt{\gcd(n,m,k)}  d(k)\sqrt k.
	\]
\end{lem}
\subsection{Bessel functions}\hspace{0cm}
Recall that the $I$-Bessel function has the well-known integral representation for $\kappa\in\R^+$ and $x\geq 0$ (see \cite[p. 79]{Watson})
\begin{equation}\label{E:int}
	I_{\kappa}(x)=\frac{\left(\frac{x}{2}\right)^{\kappa}}{\sqrt{\pi}\Gamma\left(\kappa+\frac{1}{2}\right)} \int_{-1}^1\left(1-u^2\right)^{\kappa-\frac{1}{2}}e^{xu} du.
\end{equation}

We require the following bounds.
\begin{lem}\label{lem:BesselBounds}\hspace{0cm}
	\begin{enumerate}
		\item
		For $0\leq x<1$, we have
		\[
			I_{1}(x)\leq x.
		\]
		\item If $x\geq 1$, then
		\begin{equation*}
			I_1(x) \le \sqrt{\frac2{\pi x}} e^x.
		\end{equation*}
		\item If $x\geq 3$, then
		\[
			I_1(x) \ge \frac{e^x}{4\sqrt x}.
		\]
	\end{enumerate}
\end{lem}
\begin{proof}\hspace{0cm}
The claims follow from equation \eqref{E:int} and are well-known. We give the details for (3) for the convenience of the reader.
Since the integral in \eqref{E:int} from $-1$ to $0$ is nonnegative, we have, making the change of variables $u\mapsto 1-u$,
		\begin{align*}
			I_1(x)\geq \frac{xe^x}{\pi} \int_0^1 \sqrt{2-u} \sqrt u e^{-xu} du\geq\frac{xe^x}{\pi} \int_0^1 \sqrt u e^{-xu} du=\frac{e^x}{\pi\sqrt{x}}\left(\frac{\sqrt\pi}2-\Gamma\left(\frac32,x\right)\right).
		\end{align*}
Here, for $x> 0$, the \begin{it}incomplete gamma function\end{it} is defined by $\Gamma(s,x):=\int_{x}^{\infty} t^{s-1}e^{-t}dt$.
Using  \cite[8.8.2]{NIST} and \cite[8.10.10]{NIST}, we then obtain, for $x\geq 3$
		\[
			I_{1}(x)\geq \frac{e^x}{\pi\sqrt{x}} e^{x}\left(\frac{\sqrt{\pi}}{2}-\frac{\pi\sqrt x}4e^{-x}\left(\sqrt{1+\frac4{\pi x}}-1\right) -\sqrt xe^{-x}\right)\geq \frac{e^x}{4\sqrt{x}}. \qedhere
		\]
\end{proof}
\section{Exact formulas}\label{sec:exact}
In this section we determine the modularity properties of $Q_{10}$ and the principal parts of $Q_{10}$ and $Q_{10}^{-1}$ in every cusp to find exact formulas for their Fourier coefficients.

\subsection{Modularity}
Using \eqref{E:triple}, we have
\begin{equation*}
	Q_{10}(q) = q^{-1} f(\tau),
\end{equation*}
where
\begin{equation*}
	f(\t) := \frac{\vartheta(\t;10\t)}{\vartheta(3\t;10\t)}.
\end{equation*}
For $h\in\Z$, we write
\begin{equation}\label{eqn:3hsplit}
	3h=h_1k+h_2, \quad h_1\in\Z, \quad 0\le h_2<k.
\end{equation}
We set $d:=\gcd(k,10)$ throughout and define  $\nu_{1},\nu_2,\mu_{1},\mu_2 \in \Z$ with $0\le \nu_{2},\mu_2<d$ via
\begin{equation}\label{E:hn}
	h = d\nu_1+\nu_2 , \quad h_2 = d\mu_1+\mu_2.
\end{equation}
Setting $\zeta_{n}:=e^{\frac{2\pi i}{n}}$, a direct calculation using \Cref{lem:thetaprops} (1), (2) gives the following transformation law for $f$.
\begin{prop}\label{P:Trans}
Let $\re(z)>0$, $k\in\N$ and $h,h'\in\Z$ with $\gcd(h,k)=1$, $\frac{10}{d}\mid h'$, and $hh'\equiv -1\pmod{k}$. Then
\begin{multline*}
		f\left(\frac1k(h+iz)\right)=(-1)^{h_1+\nu_1+\mu_1} \zeta_{10k}^{3\mu_2-\nu_2}e^{\frac{2\pi id^2}{20k}\left(\nu_1^2- \mu_1^2\right)h'} e^{\frac\pi{10kz}\left(\mu_2^2-\nu_2^2\right)}e^{ - \frac{4\pi z}{5k}}\\
\times \frac{\vartheta\left(\frac{i\nu_2}{\frac{10}d kz} - \frac{\nu_1d^2h'}{10k} - \frac1{\frac{10}d k};\frac{d^2}{10k}\left(h'+\frac iz\right)\right)}{\vartheta\left(\frac{i\mu_2}{\frac{10}d kz} - \frac{\mu_1d^2h'}{10k} - \frac3{\frac{10}d k};\frac{d^2}{10k}\left(h'+\frac iz\right)\right)}.
\end{multline*}
\end{prop}
\subsection{Principal parts and exact formulas}
\subsubsection{$\d = 1$}
We next determine the principal parts of $f$ at each cusp in order to obtain an exact formula for $c_1(n)$. Before stating the identity, for $k\in\N$, and $j\in\Z$ with $\gcd(j,d)=1$, we define the Kloosterman-type sums
\[
	A_{k,j}(n) := \sum_{\substack{1\leq h<k\\ h\equiv j\pmod{d}\\hh'\equiv-1\Pmod k\\ \frac{10}d\mid h'}} (-1)^{h_1+\nu_1+\mu_1} \zeta_{10k}^{3\mu_2-\nu_2-d}e^{\frac{2\pi i}{k}\left( \frac{d^2}{20}\left(\nu_1^2-\mu_1^2+\nu_1-\mu_1\right)h'-(n+1)h\right)}.
\]
We first note that the above sum only depends on $h,h'\pmod{k}$ and rewrite $A_{k,j}(n)$, uniquely choosing for $j\in\Z$ and $d\mid 10$ an integer $\alpha_j(d)$ satisfying
\begin{equation}\label{eqn:alphajdef}
\alpha_j(d)\equiv 3j\pmod{d}\qquad 1\leq \alpha_j(d)<d.
\end{equation}
We omit $d$ from the notation if it is clear from the context.
\begin{lem}\label{lem:Akjrewrite}
For $k\in\N$ and $d=\gcd(k,10)$, suppose that $1\leq j<d$ with $\gcd(j,d)=1$. Then for $n\in\Z$ we have
\[
	A_{k,j}(n) = \zeta_{10k}^{3\alpha_j-j-d}\sum_{\substack{h\pmod{k}^*\\ h\equiv j\pmod{d}\\hh'\equiv-1\Pmod k\\ \frac{10}d\mid h'}} (-1)^{h_1+\nu_1+\mu_1}e^{\frac{2\pi i}{k}\left( \frac{d^2}{20}\left(\nu_1^2-\mu_1^2+\nu_1-\mu_1\right)h'-(n+1)h\right)}.
\]
\end{lem}
\begin{proof}
First note that since $1\leq j\leq d$, $1\leq \nu_2<d$, and
\begin{equation}\label{eqn:nu2j}
j\equiv h\equiv \nu_2\pmod{d}\Rightarrow \nu_2=j
\end{equation}
 by \eqref{E:hn}. Similarly, since $1\leq \alpha_j$, $\mu_2<d$ and
\begin{equation}\label{eqn:mu2alphaj}
\mu_2\equiv 3h\equiv \alpha_j\pmod{d}\Rightarrow \mu_2=\alpha_j.
\end{equation}
 Thus
\[
\zeta_{10k}^{3\mu_2-\nu_2-d}=\zeta_{10k}^{3\alpha_j-j-d}
\]
 is independent of $h$ and we can pull it outside of the sum.

	It hence remains to show that we may let $h$ and $h'$ run $\hspace{-0.1cm}\Pmod k$. We start with $h$. Clearly the factor $e^{-\frac{2\pi i(n+1)h}{k}}$ is invariant under $h\mapsto h+k$, so we check that the remaining quantities do not change as $h\mapsto h+\ell k$ for $\ell\in\Z$. Making the change of variables $h\mapsto h+\ell k$ in \eqref{eqn:3hsplit} gives
	\begin{equation*}
		3(h+\ell k)=h_1k+h_2+3\ell k=(h_1+3\ell)k+h_2.
	\end{equation*}
	Thus $h_2$ stays unchanged and $h_1\mapsto h_1+3\ell$.  Moreover, making the change of variables $h\mapsto h+\ell k$ in \eqref{E:hn} gives
	\begin{equation*}
		h+\ell k=d\nu_1+\nu_2+\ell k=d\left(\nu_1+\ell\frac kd\right)+\nu_2.
	\end{equation*}
	Thus $\nu_2$ stays unchanged and $\nu_1\mapsto\nu_1+\ell\frac kd$.

Note that since $h_2$ remains unchanged, so do  $\mu_1$ and $\mu_2$. Thus we want
\[
		(-1)^{h_1+3\ell+\nu_1+\ell\frac kd+\mu_1} e^{\frac{2\pi id^2}{20k}\big(\left(\nu_1+\ell\frac kd\right)^2-\mu_1^2+\left(\nu_1+\ell\frac kd\right)-\mu_1\big)h'} = (-1)^{h_1+\nu_1+\mu_1} e^{\frac{2\pi id^2}{20k} \left(\nu_1^2-\mu_1^2+\nu_1-\mu_1\right)h'}.
\]
	This holds if
	\begin{equation*}
		(-1)^{\ell+\ell\frac kd} e^{\frac{2\pi id^2}{20k}\left(2\ell\frac kd\nu_1+\ell^2 \frac{k^2}{d^2}+\ell\frac kd\right)h'} = 1.
	\end{equation*}
	Thus we want
	\begin{equation*}
		\frac\ell2 + \ell\frac k{2d} + \frac{d^2}{20k}\left(2\ell\frac kd\nu_1+\ell^2\frac{k^2}{d^2}+\ell\frac kd\right)h' \in\Z.
	\end{equation*}
	We now distinguish the cases $d=5$ and $d=10$.

	If $d=5$, then $2\mid h'$ and $\frac kd$ is odd. Thus we want
	\begin{equation*}
		\frac12 \ell \left(1+\frac k5\right) + \frac12\left(2\ell\nu_1+\ell^2\frac k5+\ell\right)\frac{h'}2 \equiv 0 \Pmod1 \Leftrightarrow \frac12 \ell(\ell+1) \frac{h'}2 \equiv 0 \Pmod1,
	\end{equation*}
	which holds. 

	If $d=10$, then we want
	\begin{equation*}
		\frac\ell2 + \frac{\ell k}{20} + \frac\ell2\left(2\nu_1 + \ell\frac k{10} + 1\right)h' \in \Z.
	\end{equation*}
	Since $2\mid k$, we have $h'$ odd, so the left-hand side is
	\[
		\frac{\ell}{2} + \frac{\ell\frac{k}{10}}{2} + \frac{\ell}{2}\left(\ell\frac{k}{10}+1 \right)\pmod{1}.
	\]
	If $\ell$ is even, then every term is an integer, so the claim follows. If $\ell$ is odd, then we have
	\[
		\frac{\ell}{2} + \frac{\ell\frac{k}{10}}{2} + \frac{\ell}{2}\left(\ell\frac{k}{10}+1 \right)\equiv \frac12+\frac{\frac{k}{10}}{2} + \frac{\frac{k}{10}}{2}+\frac12 \equiv 0\pmod{1}.
	\]
	We hence conclude that the sum only depends on $h$ (mod $k$).

	To see the invariance under $h'\Pmod k$, we again split into cases depending on $d$. For $d=5$, we change $h'$ into $h'+2\ell k$ to preserve the condition $\frac{10}{d}\mid h'$. We want
	\begin{equation*}
		e^{5\pi i\left(\nu_1^2-\mu_1^2+\nu_1-\mu_1\right)\ell} = 1.
	\end{equation*}
	This holds because $2\mid (\nu_1^2-\mu_1^2+\nu_1-\mu_1)$.

	For $d=10$, we change $h'$ into $h'+\ell k$ and we want
	\begin{equation*}
		e^{10\pi i\left(\nu_1^2-\mu_1^2+\nu_1-\mu_1\right)\ell} = 1.
	\end{equation*}
	This also holds.
\end{proof}
Abbreviating
\begin{equation}\label{eqn:Akdef}
A_k(n):=\sum_{\pm} A_{k,\pm 3}(n),
\end{equation}
we next give the exact formula for $c_1(n)$.
\begin{lem}\label{lem:exact1}
For $n\in\N$, we have
	\begin{equation*}
		c_1(n) = \frac{\sqrt2\pi}{\sqrt{5n+8}} \sum_{\substack{k\ge1\\ d=\gcd(k,10)\in\{5,10\}}} \frac{\sqrt{d-4}}k  A_k(n)I_1\left(\frac{2\pi}{5k}\sqrt{2(d-4)(5n+8)}\right).
	\end{equation*}
\end{lem}
\begin{proof}
	Using \Cref{lem:thetaprops} (3) with $a = \frac{\nu_2}{d}$ and $a = \frac{\mu_2}{d}$, we see by \Cref{P:Trans} that there can only be growth as $z\to 0$ if there is growth in the factor
	\begin{equation*}
		e^{\frac\pi{10kz}\left(\mu_2^2-\nu_2^2\right)+\frac\pi{\frac{10}d kz}(\nu_2-\mu_2)}= e^{\frac\pi{10kz}\left(\mu_2^2-\nu_2^2+d\left(\nu_2-\mu_2\right)\right)}.
	\end{equation*}
Thus, in order to have growth, we require
	\begin{equation*}
		 \mu_2^2-\nu_2^2 + d(\nu_2 - \mu_2) > 0.
	\end{equation*}
	A direct calculation yields that this is satisfied if and only if ($d = 5$ and $\nu_2 \in \{2,3\}$) or ($d = 10$ and $\nu_{2} \in \{3,7\}$). Equivalently we may describe this as $d = 5$ and $h \equiv \pm 2 \pmod{5}$ or $d = 10$ and $h \equiv \pm 3 \pmod{10}$.

	Noting that the error term in \Cref{lem:thetaprops} (3) does not contribute to the principal part, we obtain as principal part
	\begin{equation}\label{E:PP}
		(-1)^{h_1+\nu_1+\mu_1} \zeta_{10k}^{3\mu_2-\nu_2-d} e^{\frac{2\pi id^2}{20k} \left(\nu_1^2-\mu_1^2+\nu_1-\mu_1\right)h'} e^{\frac{\pi}{10kz}\left(\mu_2^2-\nu_2^2+d\left(\nu_2-\mu_2\right)\right)}e^{-\frac{4\pi z}{5k}}.
	\end{equation}
	Plugging \eqref{E:PP} into\footnote{Note the shift $n\mapsto n+1$ because we have a $q^{-1}$ outside.} \Cref{Thm:Zuckerman}, we have
	\begin{multline*}
		c_1(n)=\frac{\sqrt2\pi}{\sqrt{5n+8}} \sum_{\substack{k\ge1\\ d=\gcd(k,10)\in\{5,10\}}} \frac{\sqrt{d-4}}k\sum_{\substack{1\le h\le k\\\gcd(h,k)=1\\h\equiv \pm 3\pmod{d}\\hh'\equiv-1\Pmod k\\\frac{10}d\mid h'}} (-1)^{h_1+\nu_1+\mu_1}\zeta_{10k}^{3\mu_2-\nu_2-d}\\
		\times e^{\frac{2\pi i}{k}\left(\frac{d^2}{20}\left(\nu_1^2 - \mu_1^2 + \nu_1 - \mu_1\right)h'-(n+1)h\right)}I_1\left(\frac{2\pi}{5k}\sqrt{2(d-4)(5n+8)}\right).
	\end{multline*}
	Plugging in \eqref{eqn:Akdef} and the definition of $A_{k,\pm3}(n)$ gives the claim.
\end{proof}
\subsubsection{$\d = -1$}

To state the exact formula in this case, we abbreviate
\begin{equation}\label{E:MA}
	\CA_k(n):=\sum_{\pm}\overline{A_{k,\pm 1}(-n)}.
\end{equation}
\begin{lem}\label{lem:exact-1}
For $n\in\N$, we have
	\[
		c_{-1}(n)=\frac{\sqrt2\pi}{\sqrt{5n-8}} \sum_{\substack{k\ge1\\ d=\gcd(k,10)\in\{5,10\}}} \frac{\sqrt{d-4}}{k} \CA_k(n) I_1\left(\frac{2\pi}{5k}\sqrt{2(d-4)(5n-8)}\right).
	\]
\end{lem}
\begin{proof}
We first determine which cusps $\frac{h}{k}$ contribute a non-trivial principal part. We get a contribution if ($d=5$ and $\nu_2\in\{1,4\}$) or if ($d=10$ and $\nu_2\in\{1,9\}$), and a straightforward calculation yields that the principal parts at the corresponding cusps are
\[
	(-1)^{h_1+\nu_1+\mu_1} \zeta_{10k}^{d+\nu_2-3\mu_2}e^{\frac{2\pi id^2}{20k}\left(\mu_1^2-\nu_1^2+\mu_1-\nu_1\right)h'} e^{\frac{\pi}{10kz}\left(\nu_2^2-\mu_2^2+d\left(\mu_2-\nu_2\right)\right)} e^{\frac{4\pi z}{5k}}.
\]
\Cref{Thm:Zuckerman} then implies that
\begin{align*}
		c_{-1}(n)
		&=\frac{\sqrt2\pi}{\sqrt{5n-8}} 
		\sum_{\substack{k\ge1\\ d=\gcd(k,10)\in\{5,10\}}} 
		\frac{\sqrt{d-4}}k\\
		&\hspace{-1cm}\times \hspace{-0.6cm}\sum_{\substack{1\le h\le k\\\gcd(h,k)=1\\h\equiv\pm1\Pmod5\\hh'\equiv-1\Pmod k\\\frac{10}d\mid h'}} \hspace{-0.8cm} (-1)^{h_1+\nu_1+\mu_1} \zeta_{10k}^{\nu_2+d-3\mu_2} e^{\frac{2\pi i}{k}\left(\frac{d^2}{20}\left(\mu_1^2 - \nu_1^2 + \mu_1 - \nu_1\right)h'-(n-1)h\right)}
 I_1\!\left(\frac{2\pi}{5k}\sqrt{2(d-4)(5n-8)}\right).
	\end{align*}
Plugging in \eqref{eqn:Akdef} and the definition of $A_{k,\pm1}(n)$ gives the claim.
\end{proof}

\section{Main term contributions}\label{sec:mainterm}
In this section we split $c_1(n)$ and $c_{-1}(n)$ into a main plus an error term.
\subsection{$\d=1$}\label{sec:maindelta=1}
We define
\begin{align*}
	M_1(n)&:=\frac{2\sqrt3\pi}{5\sqrt{5n+8}}  \cos\left(2\pi\left(\frac{4}{25}+\frac{3n}{10}\right)\right)I_1\left(\frac{2\pi}{25}\sqrt{3(5n+8)}\right),\\
	E_1(n)&:=\frac{\sqrt2\pi}{\sqrt{5n+8}} \sum_{\substack{k\ge1\\ d=\gcd(k,10)\in\{5,10\}\\ k\neq 10}} \frac{\sqrt{d-4}}k  A_k(n)I_1\left(\frac{2\pi}{5k}\sqrt{2(d-4)(5n+8)}\right).
\end{align*}
\begin{lem}\label{lem:main+error1}
For $n\in\N$, we have
	\[
		c_{1}(n)=M_1(n)+E_1(n).
	\]
	Moreover, if $|E_1(n)|<|M_1(n)|$, then $\sgn(c_1(n))$ agrees with what is claimed in \Cref{C:MainConj} for these $n$.
\end{lem}
\begin{proof}
By \Cref{lem:exact1}, the identity is equivalent to proving that $M_1(n)$ is the term $k=10$. Namely, we need to show that
\begin{equation}\label{eqn:M1(n)}
M_1(n)=\frac{\sqrt{3}\pi}{5\sqrt{5n+8}} A_{10}(n)I_{1}\left(\frac{2\pi}{25}\sqrt{3(5n+8)}\right).
\end{equation}
Plugging in \Cref{lem:Akjrewrite} and recalling the definition \eqref{eqn:alphajdef} to simplify $A_{10}(n)$, the right-hand side of \eqref{eqn:M1(n)} becomes
\[
	\frac{\sqrt3\pi}{5\sqrt{5n+8}} I_{1}\left(\frac{2\pi}{25}\sqrt{3(5n+8)}\right) \sum_{h \in \{3,7\}} (-1)^{h_{1}}\zeta_{100}^{3\alpha_h(10) - h - 10}e^{-\frac{2\pi i(n+1)h}{10}}.
\]
Noting that the sum on $h$ equals $2\cos(2\pi(\frac4{25}+\frac{3n}{10}))$ yields \eqref{eqn:M1(n)}, and hence the claimed identity.

Now assume that $|E_1(n)|<|M_1(n)|$. In this case, we conclude that
\[
	\sgn\left(c_1(n)\right)=\sgn\left(M_1(n)\right)=\sgn\left(\cos\left(2\pi\left(\frac{4}{25}+\frac{3n}{10}\right)\right)\right).
\]
This gives the sign pattern $+-++--+--+$, matching what is claimed.
\end{proof}

\subsection{$\d=-1$}
We define
\begin{align*}
	M_{-1}(n)&:=\frac{2\sqrt3\pi}{5\sqrt{5n-8}}  \cos\left(2\pi\left(\frac{3}{25}-\frac{n}{10}\right)\right)I_1\left(\frac{2\pi}{25}\sqrt{3(5n-8)}\right),\\
	E_{-1}(n)&:=\frac{\sqrt2\pi}{\sqrt{5n-8}} \sum_{\substack{k\ge1\\ d=\gcd(k,10)\in\{5,10\}\\ k\neq 10}} \frac{\sqrt{d-4}}{k} \CA_k(n) I_1\left(\frac{2\pi}{5k}\sqrt{2(d-4)(5n-8)}\right).
\end{align*}
As before, we obtain the following.
\begin{lem}\label{lem:main+error-1}
For $n\in\N$, we have
	\[
		c_{-1}(n)=M_{-1}(n)+E_{-1}(n).
	\]
	Moreover, if $|E_{-1}(n)|<|M_{-1}(n)|$, then $\sgn(c_{-1}(n))$ agrees with what is claimed in \Cref{C:MainConj} for this $n$.
\end{lem}

\section{Kloosterman sums bounds}\label{sec:Kloosterman}
\subsection{$d=5$}
\begin{lem}\label{L:d=5gen}
If $k\in\N$ with $d=\gcd(k,10)=5$ and $j\in\Z$ with $\gcd(j,5)=1$, then for $n\in\Z$ we have
\rm
	\begin{equation*}
			\left|A_{k,j}(n)\right| \le 2d(k)\sqrt{\frac k5}.
		\end{equation*}
\end{lem}
\begin{proof}
Since $A_{k,j}$ only depends on $j$ (mod $5$), we assume without loss of generality that $1\leq j\leq 4$.
By Lemma \ref{lem:Akjrewrite}, we have
\[
	A_{k,j}(n) = \zeta_{10k}^{3\alpha_j-j-5}\sum_{\substack{h\pmod{k}^*\\ h\equiv j\pmod{5}\\hh'\equiv-1\Pmod k\\ 2\mid h'}} (-1)^{h_1+\nu_1+\mu_1}e^{\frac{2\pi i }{k}\left( \frac{5}{4}\left(\nu_1^2-\mu_1^2+\nu_1-\mu_1\right)h'-(n+1)h\right)}.
\]
	Note that by \eqref{eqn:mu2alphaj} $\mu_2=\alpha_j$ only depends on $j$. Rearranging \eqref{E:hn}, we have, using \eqref{eqn:3hsplit},
	\begin{equation*}
		\nu_1 = \frac{h-j}5,
		\quad  \mu_{1} = \frac{3h-h_{1}k-\alpha_j}{5}.
	\end{equation*}
	Note that, using \eqref{eqn:3hsplit} and \eqref{E:hn},
	\begin{equation*}
		h_1 \equiv h + h_2 \Pmod2, \quad \mu_1\equiv h_2 + \alpha_j \Pmod2, \quad \nu_1 \equiv h+j \Pmod2.
	\end{equation*}
	Thus
	\begin{equation*}
		h_1 + \nu_1 + \mu_1 \equiv j+\alpha_j\pmod{2}.
	\end{equation*}
	This yields that
\begin{equation}\label{eqn:Akjnsimp}
	A_{k,j}(n) = (-1)^{j+\alpha_j}\zeta_{10k}^{3\alpha_j-j-5}\sum_{\substack{h\pmod{k}^*\\ h\equiv j\pmod{5}\\hh'\equiv-1\Pmod k\\ 2\mid h'}} e^{\frac{2\pi i }{k}\left( \frac{5}{4}\left(\nu_1^2-\mu_1^2+\nu_1-\mu_1\right)h'-(n+1)h\right)}.
\end{equation}
	We next simplify $e^\frac{2\pi iA}{4k}$, where
	\begin{align*}
		A &:= 5\left(\nu_1^2-\mu_1^2+\nu_1-\mu_1\right)h' = 5\left(\left(\tfrac{h-j}5\right)^2-\left(\tfrac{3h-h_1k-\alpha_j}5\right)^2+\tfrac{h-j}5-\tfrac{3h-h_1k-\alpha_j}5\right)h'\\
		&= \tfrac15\left(-8h^2+(6\alpha_j-2j-10)h-h_1^2k^2 +6hh_1k+(5-2\alpha_j)h_1k+j^2-5j-\alpha_j^2+5\alpha_j\right)h'.
	\end{align*}

	Now we choose $h'$ with $4\mid h'$. Then $A\equiv0 \pmod{4}$. Moreover we choose $h'$ with $hh'\equiv-1\Pmod{5k}$. Then we obtain
	\begin{align*}
		A&\equiv \frac{1}{5}\left(8h-(6\alpha_j-2j-10)-6h_1k+(5-2\alpha_j)h_1kh'+\left(j^2-5j-\alpha_j^2+5\alpha_j\right)h'\right)\\
		&\equiv\frac15\left(8h+\left(j^2-5j-\alpha_j^2+5\alpha_j\right)h'-( 6\alpha_j-2j-10)\right)\\
		& = \frac45\left(2h+\frac{1}{2}\left(j^2-5j-\alpha_j^2+5\alpha_j\right)\frac{h'}{2}\right)-\frac{2}{5}( 3\alpha_j-j-5) \Pmod k.
	\end{align*}
Plugging back into \eqref{eqn:Akjnsimp}, it is not hard to see that
	\[
		A_{k,j}(n)=-\sum_{\substack{h\pmod k^*\\h\equiv j\pmod5\\hh'\equiv-1\Pmod{5k}\\4\mid h'}} e^{\frac{2\pi i}{k}\left(\frac{1}{5} \left(2h+\left(j^2-5j-\alpha_j^2+5\alpha_j\right)\frac{h'}{4}\right)-(n+1)h\right)}.
	\]
	We next make the change of variables $h'\mapsto 4h'$, and note that we may take $[4]_{5k}=\frac{1-k^2}4$. Thus we obtain
	\begin{equation*}
			A_{k,j}(n) = -\frac{1}{25} \sum_{\ell\Pmod5} e^{\frac{2\pi i j \ell}5} \sum_{\substack{h\Pmod{5k}^*\\ hh'\equiv -1\pmod{5k}}} e^{\frac{2\pi i}{5k}\left(\left((5n+3)\frac{k^2-1}{4}+\ell k\right)h+\left(j^2-5j-\alpha_j^2+5\alpha_j\right)h'\right)},
		\end{equation*}
	since
	\begin{equation*}
		\sum_{\ell\Pmod5} e^{\frac{2\pi i}5(h+j)\ell} =
		\begin{cases}
				5 & \text{if $h\equiv -j \Pmod5$},\\
				0 & \text{otherwise}.
		\end{cases}
	\end{equation*}
	Plugging in definition \eqref{eqn:Kloosterman}, we obtain
	\begin{equation*}
		A_{k,j}(n) =	-\frac{1}{25} \sum_{\ell\Pmod5} e^{\frac{2\pi i j\ell}5} K_{5k}\left((5n+3)\frac{k^2-1}4+\ell k,j^2-5j-\alpha_j^2+5\alpha_j\right).
		\end{equation*}
	We now use \Cref{lem:KloostermanBound}
and
\begin{equation}\label{eqn:d5k}
d(5k)\le d(5)d(k)=2d(k)
\end{equation}
to bound
	\begin{align*}
			&\hspace{-0.5cm}\left|K_{5k}\left((5n+3)\frac{k^2-1}4+\ell k,j^2-5j-\alpha_j^2+5\alpha_j\right)\right|\\
	&\le \sqrt{\gcd\left((5n+3)\frac{k^2-1}4+\ell k,j^2-5j-\alpha_j^2+5\alpha_j,5k\right)}  d(5k)\sqrt{5k}\\
	&\le\sqrt2\sqrt{\left|j^2-5j-\alpha_j^2+5\alpha_j\right|}  d(k)\sqrt{5k},
	\end{align*}
	where we note that $2$ cannot divide the gcd because $5k$ is odd. Thus
	\begin{equation*}
		|A_{k,j}(n)| \le \frac{\sqrt{2k}}{\sqrt5}\sqrt{\left|j^2-5j-\alpha_j^2+5\alpha_j\right|} d(k).
	\end{equation*}
	We finally compute
	\[
		j^2-5j-\alpha_j^2+5\alpha_j=
		\begin{cases}
			2 & \text{if }j\in\{1,4\},\\
			-2 & \text{if }j\in\{2,3\}.
		\end{cases}\qedhere
	\]
\end{proof}
\subsection{$d=10$}
For $d=10$, we have the following.
\begin{lem}\label{lem:d=10gen}
If $k\in\N$ with $\gcd(k,10)=10$ and $j\in\Z$ with $\gcd(j,10)=1$, then for $n\in\Z$ we have
	\begin{equation*}
		|A_{k,j}(n)| \le d(10k) \sqrt{\frac{3k}5}.
	\end{equation*}
\end{lem}
\begin{proof}
	We again assume without loss of generality that $1\leq j<10$. By \Cref{lem:Akjrewrite}, we have
	\[
		A_{k,j}(n) = \zeta_{10k}^{3\alpha_j-j-10}\sum_{\substack{h\pmod{k}^*\\ h\equiv j\pmod{10}\\hh'\equiv-1\Pmod k}} (-1)^{h_1+\nu_1+\mu_1}e^{\frac{2\pi i}{k}\left(5\left(\nu_1^2-\mu_1^2+\nu_1-\mu_1\right)h'-(n+1)h\right)}.
	\]
	We next write the factor involving $h'$ as $e^\frac{2\pi iB}k$. By \eqref{eqn:nu2j} and \eqref{eqn:mu2alphaj}, we have that $\nu_2=j$ and $\mu_2=\alpha_j$. By \eqref{eqn:3hsplit} and \eqref{E:hn}, we have
	\begin{equation*}
		 \nu_1 = \frac{h-j}{10}, \quad \mu_1 = \frac{3h-h_1k-\alpha_j}{10}.
	\end{equation*}
	Thus we have
	\begin{align*}
		B&=5\left(\nu_1^2-\mu_1^2+\nu_1-\mu_1\right)h' = 5\left(\left(\tfrac{h-j}{10}\right)^2-\left(\tfrac{3h-h_1k-\alpha_j}{10}\right)^2 + \tfrac{h-j}{10} - \tfrac{3h-h_1k-\alpha_j}{10}\right) h'\\
		&= \tfrac15\left(-2h^2 + \left(\tfrac{3\alpha_j-j}{2}-5\right) h + 3hh_1\tfrac k2 +\left(5-\alpha_j\right)h_1\tfrac{k}{2} - h_1^2\tfrac{k^2}4 +\tfrac14\left(j^2-10j-\alpha_j^2+10\alpha_j\right)\right)h'.
	\end{align*}
	We note that since $j\in\{1,3,7,9\}$, we have $\alpha_j\in\{3j,3j-20\}$, so $\alpha_j\equiv j\pmod{2}$, and hence $\frac{j^2}{4}-\frac{\alpha_j^2}{4}\in\Z$ and $\frac{j+ c\alpha_j}{2}\in\Z$ for any odd $c$. We now choose $h'$ so that $hh'\equiv-1\Pmod{10k}$, giving
	\begin{equation*}
		B \equiv  \tfrac25h +\tfrac1{10}(j-3\alpha_j)+1+ \tfrac1{20}\left(j^2-10j-\alpha_j^2+10\alpha_j\right)h' - \tfrac k{10}h_1 \left(3 +(\alpha_j-5)h' +h'\tfrac k{2}\right)\pmod{k}.
	\end{equation*}
	Next we use that $\alpha_j\equiv 3j\pmod{10}$ and $h'\equiv -[j]_{10}\pmod{10}$ to show that
	\begin{equation}\label{eqn:Bcong}
		B \equiv \tfrac25h +\tfrac1{10}(j-3\alpha_j)+1+ \tfrac1{20}\left(j^2-10j-\alpha_j^2+10\alpha_j\right)h' - \tfrac k{2}h_1[j]_{10}\left(1-\tfrac{k}{10}\right)\pmod{k}.
	\end{equation}
	We now distinguish two cases depending on whether $2\mid\mid k$ or $4\mid k$. If $2\mid\mid k$, then we have
	\[
		\frac {k}{2}h_1[j]_{10} \left(1 - \frac k{10}\right)\equiv 0\pmod{k}.
	\]
	Thus we obtain in this case
	\[
		B \equiv \frac25h +\frac1{10}(j-3\alpha_j)+1+ \frac1{20}\left(j^2-10j-\alpha_j^2+10\alpha_j\right)h'\pmod{k}.
	\]
	We next simplify the sign factor in this case. Plugging
 $\nu_2=j$ and $\mu_2=\alpha_j$ into \eqref{eqn:3hsplit} and \eqref{E:hn}, we conclude that
	\begin{equation*}
		3(10\nu_1+j) = h_1k + h_2 = h_1k + 10\mu_1 + \alpha_j.
	\end{equation*}
	Taking this modulo $4$, we obtain
	\begin{equation*}
		h_1 + \nu_1 + \mu_1 \equiv \frac12(\alpha_j+j) \Pmod2.
	\end{equation*}
	Thus the sign factor becomes $(-1)^{\frac{\alpha_j+j}{2}}$, which only depends on $j$. We may also pull out the contribution from the term $\frac1{10}(j-3\alpha_j)+1$. In this case, we obtain overall
	\begin{equation}\label{eqn:Akjn2||k}
		|A_{k,j}(n)| = \left|\sum_{\substack{h\pmod k^*\\h\equiv j\pmod{10}\\hh'\equiv-1\Pmod {10k}}} e^{\frac{2\pi  i}{5k}\left(-(5n+3)h+ \frac14 \left(j^2-10j-\alpha_j^2+10\alpha_j\right)h'\right)}\right|.
	\end{equation}

	We next consider the case $4\mid k$. In this case
	\begin{equation*}
		\frac{k}2 h_1[j]_{10}\left(1-\frac k{10}\right) \equiv \frac{k}2h_1 [j]_{10} \pmod{k}.
	\end{equation*}
	Noting that $[j]_{10}\equiv j\equiv 1\pmod{2}$ because $\gcd(j,10)=1$, this term contributes
	\begin{equation*}
		e^\frac{2\pi i\frac{k}2h_1[j]_{10}}k = (-1)^{h_1}.
	\end{equation*}
Plugging this into \eqref{eqn:Bcong}, for $4\mid k$ we have
	\[
	e^{\frac{2\pi i B}{k}}= (-1)^{h_1} e^{\frac{2\pi i}{5k}\left(2h+\frac12(j-3\alpha_j+10) + \frac14\left(j^2-10j-\alpha_j^2+10\alpha_j\right) h'\right)}.
	\]
	We again plug  $\nu_2=j$ and $\mu_2=\alpha_j$ into \eqref{eqn:3hsplit} and \eqref{E:hn} to obtain
	\begin{equation*}
		15\nu_1 - 5\mu_1 = h_1\frac k2+\frac{\alpha_j-3j}{2}.
	\end{equation*}
	Since $4\mid k$, we obtain
	\[
		\nu_1+\mu_1 \equiv \frac{\alpha_j+j}{2}\Pmod2.
	\]
 	Thus in this case,
	we obtain again \eqref{eqn:Akjn2||k}.

	We then write the Kloosterman sum as
	\begin{equation*}
		|A_{k,j}(n)| = \frac{1}{50} \left|\sum_{\ell\pmod{5}} e^{-\frac{2\pi ij\ell}{5}}K_{10k}\left(2(k\ell-5n-3),\frac12\left(j^2-10j-\alpha_j^2+10\alpha_j\right)\right)\right|.
	\end{equation*}
	Thus, using \Cref{lem:KloostermanBound},
	\begin{align*}
		|A_{k,j}(n)| &\le \frac1{10} \sqrt{\max_{0\le\ell\le4}\gcd\left(2(k\ell-5n-3),\frac12\left(j^2-10j-\alpha_j^2+10\alpha_j\right),10k\right)} d(10k) \sqrt{10k}\\
& \le \frac1{10} \sqrt{\left|j^2-10j-\alpha_j^2+10\alpha_j\right|}\sqrt{5k} d(10k).
	\end{align*}
Finally, we compute
\[
	j^2-10j-\alpha_j^2+10\alpha_j=
	\begin{cases}
		12 & \text{if }j\in\{1,9\},\\
		-12 & \text{if }j\in\{3,7\}.
	\end{cases}\qedhere
\]
\end{proof}
\section{Bounding the error terms}\label{sec:error}
\subsection{$\d=1$}
We explicitly need to bound the Kloosterman sums defined in \eqref{eqn:Akdef}. For this, we distinguish two cases.
\subsubsection{$d=5$}
First assume that $d=5$. Using \eqref{eqn:Akdef}, we obtain the following from \Cref{L:d=5gen}.
\begin{lem}\label{L:d=5}
If $k\in\N$ with $\gcd(k,10)=5$, then for $n\in\Z$
we have
	\begin{equation*}
		|A_k(n)| \le 4 d(k)\sqrt{\frac k5}.
	\end{equation*}
\end{lem}
Using \Cref{lem:exact1} and \Cref{L:d=5}, the contribution from $d=5$ can be bounded against
\begin{equation}\label{eqn:5midk1}
	\frac{4\sqrt2\pi}{5\sqrt{5n+8}} \sum_{k\ge1} \frac{d(5k)}{\sqrt k} I_1\left(\frac{2\pi}{25k}\sqrt{2(5n+8)}\right).
\end{equation}
By \Cref{lem:BesselBounds} (1) and \eqref{eqn:d5k} together with the well-known identity
\begin{equation}\label{eqn:dnDirichlet}
\sum_{n\geq 1} \frac{d(n)}{n^s}=\zeta(s)^2
\end{equation}
for $\re(s)>1$, the contribution from the terms with $k>\frac{2\pi}{25}\sqrt{2(5n+8)}$ can be bounded against
\begin{equation}\label{eqn:5midkbig1}
	\frac{16\pi^2}{125} \sum_{k>\frac{2\pi}{25}\sqrt{2(5n+8)}} \frac{d(5k)}{k^\frac32} \le \frac{32\pi^2\zeta\left(\frac32\right)^2}{125}.
\end{equation}
The remaining contribution can be estimated against
\begin{equation}\label{E:SmallBound}
	\frac{4\sqrt2\pi}{5\sqrt{5n+8}} I_1\left(\frac{2\pi}{25}\sqrt{2(5n+8)}\right) \sum_{1\le k\le\frac{2\pi}{25}\sqrt{2(5n+8)}} \frac{d(5k)}{\sqrt k}.
\end{equation}
We use \eqref{eqn:d5k} and then note that since divisors of $k$ come in pairs $(d,\frac{k}{d})$ with $\min\{d,\frac{k}{d}\}\leq \sqrt{n}$, we may trivially bound
\begin{equation}\label{eqn:d(k)bnd}
	d(k)\leq 2\sqrt{k}.
\end{equation}
So we may estimate \eqref{E:SmallBound} against
\begin{equation}\label{eqn:5midksmall1}
	\frac{64\pi^2}{125}  I_1\left(\frac{2\pi}{25}\sqrt{2(5n+8)}\right).
\end{equation}
\subsubsection{$d=10$}
Using again \eqref{eqn:Akdef} and \Cref{lem:d=10gen}, we obtain the following.
\begin{lem}\label{L:d=10}
If $k\in\N$ with $\gcd(k,10)=10$, then for $n\in\Z$ we have
	\begin{equation*}
		|A_k(n)| \le 2d(10k)\sqrt{\frac{3k}5}.
	\end{equation*}
\end{lem}

Using \Cref{L:d=10}, the absolute value of the contribution to \Cref{lem:exact1} from $10|k$, $k\geq11$ can be bounded against
\begin{equation}\label{eqn:10midkerror1first}
	\frac{6\sqrt2\pi}{5\sqrt{5n+8}} \sum_{k\ge2} \frac{d(100k)}{\sqrt k} {I_1{\left(\frac{2\pi}{25k}\sqrt{3(5n+8)}\right)}}.
\end{equation}
Plugging in \eqref{eqn:d(k)bnd} to estimate
\begin{equation*}
	d(100k) \le d(100) d(k) = 9d(k)\leq 18\sqrt{k},
\end{equation*}
we may bound \eqref{eqn:10midkerror1first} from above by
\begin{multline}\label{eqn:10midkerror1}
 \frac{108\sqrt2\pi}{5\sqrt{5n+8}}\hspace{-.05cm} \sum_{k=2}^{\left\lfloor\frac{2\pi}{25}\sqrt{3(5n+8)}\right\rfloor}\hspace{-.05cm} {I_1{\left(\frac{2\pi}{25k}\sqrt{3(5n+8)}\right)}}\\
+\frac{54\sqrt2\pi}{5\sqrt{5n+8}}\sum_{k>\frac{2\pi}{25}\sqrt{3(5n+8)}} \frac{d(k)}{\sqrt{k}} I_1\left(\frac{2\pi}{25k}\sqrt{3(5n+8)}\right).
\end{multline}

Again using \Cref{lem:BesselBounds} (1) and \eqref{eqn:dnDirichlet}, the part from $k>\frac{2\pi}{25}\sqrt{3(5n+8)}$ contributes
\begin{equation}\label{eqn:10midkbig1}
	\frac{54\sqrt2\pi}{5\sqrt{5n+8}} \sum_{k>\frac{2\pi}{25}\sqrt{3(5n+8)}} \frac{2\pi d(k)}{25k^{\frac{3}{2}}}\sqrt{3(5n+8)} \leq  \frac{108\sqrt6\pi^2}{125} \zeta\left(\frac32\right)^2.
\end{equation}
Next, we bound the remaining contribution against
\begin{equation}\label{eqn:10midkmoderate1}
\frac{108\sqrt2\pi}{5\sqrt{5n+8}} \sum_{2\leq k\leq \frac{2\pi}{25}\sqrt{3(5n+8)}} {I_1{\left(\frac{2\pi}{25k}\sqrt{3(5n+8)}\right)}}\leq \frac{216\sqrt6\pi^2}{125} I_1\left(\frac{\pi}{25}\sqrt{3(5n+8)}\right).
\end{equation}
\subsubsection{Overall bound}\label{sec:Overallbound1}
Plugging \eqref{eqn:5midkbig1} and \eqref{eqn:5midksmall1} into \eqref{eqn:5midk1} and \eqref{eqn:10midkbig1} and \eqref{eqn:10midkmoderate1} into \eqref{eqn:10midkerror1} and then taking the sum, we obtain a bound for $|E_1(n)|$. Hence, by Lemma \ref{lem:main+error1}, we conclude that $\sgn(c_1(n))$ agrees with its claimed value in \Cref{C:MainConj} if
\begin{align}
	&\left(\frac{2\sqrt{3}\pi }{5\sqrt{5n+8}} \left|\cos\left(2\pi \left(\frac{4}{25}+\frac{3n}{10}\right)\right)\right| I_{1}\left(\frac{2\pi}{25}\sqrt{3(5n+8)}\right)\right)^{-1}
	\left(\frac{32\pi^{2} \zeta\left(\frac{3}{2}\right)^{2}}{125} \right.\nonumber\\
	&+ \left. \frac{64\pi^{2}}{125} I_{1}\left(\frac{2\pi}{25}\sqrt{2(5n+8)}\right) + \frac{108\sqrt{6}\pi^{2}\zeta\left(\frac{3}{2}\right)^{2}}{125} + \frac{216\sqrt{6}\pi^{2}}{125}I_{1}\left(\frac{\pi}{25}\sqrt{3(5n+8)}\right)\right)<1.\hspace{-0.2cm}\label{E:ToShow}
\end{align}
By Lemma \ref{lem:BesselBounds} (3), if $\frac{2\pi}{25}\sqrt{3(5n+8)}\ge 3$ (which is true for $n\ge 8$), then
\begin{equation*}
	I_1\left(\frac{2\pi}{25}\sqrt{3(5n+8)}\right) \ge \frac{5e^{\frac{2\pi}{25}\sqrt{3(5n+8)}}}{4\sqrt{2\pi}\cdot3^\frac14(5n+8)^\frac14}.
\end{equation*}
By Lemma \ref{lem:BesselBounds} (2), if $\frac{2\pi}{25}\sqrt{2(5n+8)}\ge1$ (which holds for all $n\in\N$), then
\begin{equation*}
	I_1\left(\frac{2\pi}{25}\sqrt{2(5n+8)}\right) \le \frac{5e^{\frac{2\pi}{25}\sqrt{2(5n+8)}}}{2^\frac14\pi(5n+8)^\frac14}.
\end{equation*}
Finally, if $\frac\pi{25}\sqrt{3(5n+8)}\ge1$ (which holds for $n\geq 3$), then Lemma \ref{lem:BesselBounds} (2) implies that
\begin{equation*}
	I_1\left(\frac\pi{25}\sqrt{3(5n+8)}\right) \le \frac{5\sqrt2e^{\frac\pi{25}\sqrt{3(5n+8)}}}{3^\frac14\pi(5n+8)^\frac14}.
\end{equation*}
So, for $n\geq 8$, we may bound the left-hand side of \eqref{E:ToShow} against
\begin{multline*}
	\frac{2\sqrt2(5n+8)^\frac34}{3^\frac14\sqrt\pi\left|\cos\left(2\pi\left(\frac4{25}+\frac{3n}{10}\right)\right)\right|}e^{-\frac{2\pi}{25}\sqrt{3(5n+8)}}\\
	\times \left(\left(8+27\sqrt6\right)\frac{4\pi^2\zeta\left(\frac{3}{2}\right)^2}{125}+\frac{32\cdot2^{\frac34}\pi e^{\frac{2\pi}{25}\sqrt{2(5n+8)}}}{25(5n+8)^\frac14}+\frac{432 \cdot3^\frac14\pi e^{\frac\pi{25}\sqrt{3(5n+8)}}}{25(5n+8)^\frac14}\right).
\end{multline*}
Now
\begin{equation*}
	\left|\cos\left(2\pi\left(\frac4{25}+\frac{3n}{10}\right)\right)\right| \ge \left|\cos\left(\frac{13\pi}{25}\right)\right|.
\end{equation*}
Thus we want
\begin{multline*}
	\frac{2\sqrt2(5n+8)^\frac34}{3^\frac14\sqrt\pi\left|\cos\left(\frac{13\pi}{25}\right)\right|}e^{-\frac{2\pi}{25}\sqrt{3(5n+8)}}\\
	\times \left(\left(8+27\sqrt6\right)\frac{4\pi^2\zeta\left(\frac{3}{2}\right)^2}{125}+\frac{32\cdot2^\frac34\pi e^{\frac{2\pi}{25}\sqrt{2(5n+8)}}}{25(5n+8)^\frac14}+\frac{432\cdot3^\frac14\pi e^{\frac\pi{25}\sqrt{3(5n+8)}}}{25(5n+8)^\frac14}\right)<1.
\end{multline*}
One can show that this holds for $n\ge 2929$. Using a computer, the inequality claimed in \Cref{C:MainConj} has been numerically confirmed for $n\leq 2928$ (other than the noted exceptions where the coefficient vanishes), and hence we may conclude the claim for all $n\in\N$.
\subsection{$\d=-1$}
\subsubsection{$d=5$}
Using \eqref{E:MA} and \Cref{L:d=5gen}, we directly obtain the following.
\begin{lem}\label{lem:calAd=5}
If $k\in\N$ with $\gcd(k,10)=5$, then for $n\in\Z$ we have
	\begin{equation*}
	|\CA_k(n)| \le 4 d(k)\sqrt{\frac k5}.
	\end{equation*}
\end{lem}
Plugging in Lemma \ref{lem:calAd=5}, the overall contribution to $E_{-1}(n)$ from $d=5$ can be bounded against
\begin{equation*}
	\frac{4\sqrt2\pi}{5\sqrt{5n-8}} \sum_{k\ge1} \frac{d(5k)}{\sqrt k}  I_1\left(\frac{2\pi}{25k}\sqrt{2(5n-8)}\right).
\end{equation*}
Now we can proceed exactly as for $\delta=1$ (just changing $5n+8$ into $5n-8$) to obtain that this may be estimated against
\begin{equation*}
	\frac{32\pi^2\zeta\left(\frac32\right)^2}{125} + \frac{64\pi^2}{125} I_1\left(\frac{2\pi}{25}\sqrt{2(5n-8)}\right).
\end{equation*}
\subsubsection{$d=10$}
Again using \eqref{E:MA} and \Cref{lem:d=10gen}, we obtain the following.
\begin{lem}
If $k\in\N$ with $\gcd(k,10)=10$, then for $n\in\Z$ we have
	\begin{equation*}
		|\CA_k(n)| \le 2 d(10k)\sqrt{\frac{3k}5}.
	\end{equation*}
\end{lem}
Overall we obtain that the contribution with $10\mid k$, $k\ge11$ can bounded against
\begin{equation*}
	\frac{6\sqrt{2}\pi}{5\sqrt{5n-8}} \sum_{k\ge2} \frac{d(100k)}{\sqrt k}  I_1\left(\frac{4\pi}{50k}\sqrt{3(5n-8)}\right).
\end{equation*}
Exactly as in the case $\delta=1$ this may be estimated against
\begin{equation*}
	\frac{108\sqrt6\pi^2}{125} \zeta\left(\frac32\right)^2 + \frac{216\sqrt6\pi^2}{125} I_1\left(\frac\pi{25}\sqrt{3(5n-8)}\right).
\end{equation*}
\subsubsection{Overall bound}\label{sec:Overallbound-1}
As above we obtain that we want for $n\geq12$
\begin{multline*}
	\frac{2\sqrt2(5n-8)^\frac34}{3^\frac14\sqrt\pi\left|\cos\left(\frac{14\pi}{25}\right)\right|}e^{-\frac{2\pi}{25}\sqrt{3(5n-8)}}\\
\times  \left(\left(8+27\sqrt6\right)\frac{4\pi^2}{125}\zeta\left(\frac{3}{2}\right)^2+\frac{32\cdot2^\frac34\pi e^{\frac{2\pi}{25}\sqrt{2(5n-8)}}}{25(5n-8)^\frac14}+\frac{432\cdot3^\frac14\pi e^{\frac\pi{25}\sqrt{3(5n-8)}}}{25(5n-8)^\frac14}\right)<1.
\end{multline*}
It is not hard to check that this holds for $n\geq 2234$. Using a computer, the inequality claimed in \Cref{C:MainConj} has been numerically verified for $n\leq2233$ (up to the noted exceptions where the coefficient vanishes), and hence we may conclude the claim for all $n\in\N$.


\end{document}